\documentclass[12pt,a4paper]{article}

\usepackage[T2A]{fontenc}
\usepackage{indentfirst}
\usepackage{misccorr}
\usepackage{graphicx}
\usepackage{amsmath,amssymb,amsthm}

\newtheorem{theorem}{Theorem}
\newtheorem{lemma}{Lemma}

\newcommand\tig{\tilde g}

\title{Helson zeta functions for characters with finitely many values}
\author{I. Bochkov \thanks{bv1997@ya.ru}\\
\\
Department of Mathematics and Computer Science,\\
St Petersburg State University, Russia}
\date{}

\begin{document}


\maketitle
\begin{abstract} We show that the analytic continuations of Helson zeta functions $  \zeta_\chi (s)= \sum_1^{\infty}\chi(n)n^{-s} $ can have essentially arbitrary poles and zeroes  in the strip $ 21/40 < \Re s < 1 $ (unconditionally), and in the whole critical strip $ 1/2 < \Re s <1 $ under Riemann Hypothesis for the function $ \chi $ taking values in cubic roots of unity. If the sets are symmetric with respect to the real axis, the same can be achieved with $ \chi $ taking values $ \pm 1 $. 
\end{abstract}

Let $ \chi \colon \mathbb N \to \mathbb T $, $ \mathbb T = \{ z \in \mathbb C \colon |z|=1 \}$, be a totally multiplicative function. The Helson zeta function $ \zeta_\chi $ is defined  as follows 
\begin{equation}\label{H} \zeta_\chi (s)= \sum_1^{\infty}\chi(n)n^{-s} . \end{equation}  

With this definition, $ \zeta_\chi $ is an analytic function in the halfplane $ \Re s > 1 $. It satisfies the Euler product formula, \[ \zeta_\chi (s)=\prod_p \frac{1}{1-\chi(p)p^{-s}} . \] 
In particular, the function $ \zeta_\chi $ has no zeroes for $ \Re s > 1 $. 

The interest to this particular generalization of the Riemann zeta function stems from a result of H. Helson \cite{Hels} asserting that for almost all $ \chi$'s the function $ \zeta_\chi $ extends analytically to the halfplane $ \Re s > 1/2 $ and the extension has no zeroes in this halfplane. Almost all here refers to the measure on the functions $ \chi $ induced by the standard product measure on the infinite dimensional torus $ \mathbb T^\infty $ via identification of $ \chi $ with the sequence $ \{ \chi (p) \} \in \mathbb T^\infty $ of its values at primes. 

In the sense of Helson's theorem, the Riemann zeta function  is in the exceptional set because of its pole at $ 0 $. This calls for further study of analytic continuation of Helson zeta function for $ \chi $'s from the exceptional set. In \cite{S} it was shown that the set of zeroes of the meromorophic continuation of $ \zeta_\chi $ in the strip $ 1/2 < \Re s<1$ can be essentially arbitrary. 

\begin{theorem}{\cite[Theorem 1.4]{S}}
For any set $ \mathfrak O $ in the strip $ 1/2  < \Re s<  39/40 $ which has no accumulation points off the line $ \Re s = 1/2 $ there exists a completely multiplicative function $ \chi $ such that the Helson zeta function $ \zeta_\chi $ admits meromophic extension to the halfplane $ \Re s > 1/2 $, and 
\[ \left\{ s : \zeta_\chi ( s ) = 0,   \Re s > \frac 12 \right\} = \mathfrak O . \] 

If the RH is assumed then the same assertion holds with $ 1$ in the place of $ 39/40 $.
\end{theorem}

The proof of this theorem in \cite{S} gives $ \zeta_\chi $ in the form of a product of two functions, $ \zeta_\chi = r_1 r_2 $. The function $r_1 $ is determined by the values of $ \chi $ on a subset, $ \mathcal P $, of primes. The set $ \mathcal P $ and the restriction $ \left. \chi\right|_{ \mathcal P } $ are defined explicitly in such a way that $ r_1 $ has the required analytic extension with the zero set $ \mathfrak O $, while $ r_2 $ is determined by the values of $ \chi $ one the rest of the primes and has zero-free analytic extension up to the critical line $ \Re s = 1/2 $. Notice that the function $ r_2 $ is constructed by probabilistic methods of the Helson's theorem rendering the overall assertion non-constructive. The construction of $ r_1 $ uses "dipoles" of the form $ ( z - \rho )^{ -1 } - ( z - \rho^\prime )^{-1} $, where $ \rho \in \mathfrak O $ and $ \rho^\prime $ is a pole nearby $ \rho $, as building blocks for its logarithmic derivative. The set of poles occuring in this way is not controlled. In the following theorem we show that the sets of zeroes and poles can be both made arbitrary given sets, and, moreover this can achieved by a character $ \chi $ taking just three values. If we add the assumption of symmetricity for the sets of poles and zeroes, then the result holds with $ \chi $ taking just two values, $ \pm 1 $.  The theorem is as follows.

To account for multiple zeroes and poles we define a multiset as a pair made of a subset $ S \subset \mathbb C $ and a function $ S \longrightarrow \mathbb N $. When we say that the set of zeroes of an analytic function $ f$ coincides with a multiset, $ M_S $, we mean that $ f ( s ) = 0 $ iff $ s \in S $, and the multiplicity of zero at $ s \in S $ is $ m ( s ) $. A similar terminology is used for poles.

\begin{theorem}
Let $ Z $ and $ P $ be arbitrary disjoint multisets $  Z$ and $ P $ in the strip $ 21/40 < \Re s<1$ having no accumulation points off the line $ \Re s = 21/40  $. 
 
1. There exists a completely multiplicative function $ \chi \colon \mathbb N \longrightarrow \{ e^{ \pm 2\pi i /3 } , 1 \} $ such that $ \zeta_\chi $ admits meromorphic continuation to the halfplane $ \Re s > 21/40 $ with $ Z $ the set of its zeroes and $ P $ the set of its poles. If the Riemann Hypothesis holds then the assertion stays true with $ 21/40 $ replaced by $ 1/2 $.
 
2. If additionally the multisets $ Z$ and $ P $ are symmetric with respect to the real axis then there exists a completely multiplicative function $ \chi \colon \mathbb N \longrightarrow \{ \pm 1 \} $ such that $ \zeta_\chi $ admits meromorphic continuation to the halfplane $ \Re s > 21/40 $ with $ Z $ the set of its zeroes and $ P $ the set of its poles. If the Riemann Hypothesis holds then the assertion stays true with $ 21/40 $ replaced by $ 1/2 $. 
\end{theorem}

Notice that the symmetricity assumption is obviously necessary in the case of $ \chi $ taking values $ \pm 1 $, for the function $ \zeta_\chi $ is real on the real line in this situation. 

Theorem 2 strengthens our result from \cite{BR} where the existence of characters $ \chi \colon \mathbb N \to \mathbb T $ with the given sets of zeroes and poles was established, in two directions. The first improvement, already mentioned, is that it is now possible to get the result with $ \chi $'s taking just two or three values (depending on whether the symmetry assumption is imposed on $ Z $ and $ P $). The second, perhaps more important, novelty is that the proof is constructive. We do not use any probabilistic arguments and present an algorithm, determining $ \chi ( p ) $ for each prime $ p $.  

The first part of Theorem 2 holds, with the same proof, if the set $ \{ e^{ \pm 2\pi i /3 } , 1 \} $ in the formulation is replaced with $ \{ z\colon  z^l = 1 \} $, for any $ l > 3 $. 

The proof of Theorem 2 is based on (i) a novel construction of the character $ \chi $ which allows to approximate the Mellin transform of a function using the logarithmic derivative of $ \zeta_\chi $, (ii) a special interpolation procedure combined with the elementary Hardy classes theory allowing to control poles. 

Let us notice that an analog of the Helson theorem can be shown to hold for the measure generated by purely discrete measure on $ \mathbb T $ supported on roots of unity and giving the equal weight to every root. The characters $ \chi $ constructed in Theorem 2 thus lie in the exceptional set even with respect to this finer measure.

Throughout the paper, all the sums with summation variable $ p $ range over primes in an interval specified in the sum limits. No limits means that the sum ranges over all primes.

\medskip 
\textit{Acknowledgements.} We are indebted to R. Romanov for infinite help in work, K. Seip who suggested the question about poles, and last, but not least, F. Petrov for asking a question about characters with $ \pm 1 $ values. 

\bigskip 

\bfseries 
\begin{center}
Logarithmic derivative
\end{center}
\mdseries

Let $ \zeta_\chi $ be the function we seek to find. Consider $g(s)=\frac{\zeta_\chi'(s)}{\zeta_\chi (s)}$. From the Euler product representation we have 
 \begin{align*} g(s)=-\sum_{ n=1 }^\infty  \chi(n)\Lambda(n)n^{-s} = -\sum_{p,a} \chi(p^a)\Lambda(p^a)p^{-as} = \\ -\sum_p \chi(p)\Lambda(p)p^{-s} - \sum_{p, a \ge 2} \chi(p^a )\Lambda(p^a)p^{-as}, \end{align*} where $p$ ranges over the primes, $ a $ over the naturals, and $ \Lambda $ is the von Mangoldt function. The second sum in the rhs is absolutely convergent for $ \Re s > 1/2 $, hence the poles and residues of $ g $ are exactly those of the function \begin{equation}\label{tgg} \tig (s)= -\sum_{p}  \chi(p)p^{-s} \log p . \end{equation}

Thus the problem is reduced to the one of constructing the function $ \tig $ with the required poles and residues. We are going to seek the function $\tig $ in the form

\[ \tig (s) = h(s)+\int_1^{\infty}q(x)x^{-s} dx, \Re s>1 , \] where $h$ is analytic in the halfplane $ \Re s> 1/2 $, and $ q ( x ) = o ( 1) $, $ x \to +\infty $. 

\begin{lemma}
Let $ q $ be a continuous function on $ [1, +\infty ) $, and $ q ( x ) = o ( 1) $, $ x \to +\infty $. 
Then there exists a completely multiplicative function $\chi \colon \mathbb N \longrightarrow \{ e^{ \pm 2\pi i /3 } , 1 \} $  such that the function
\[  \int_1^{\infty}q(x)x^{-s} dx - \sum_{p} \chi(p)p^{-s} \log p, \]
initially defined in the halfplane $ \Re s > 1 $, extends analytically to the halfplane $ \Re s > 21/40 $ unconditionally, and to the halfplane $ \Re s > 1/2 $ if the RH holds.

If additionally the function $ q $ is real, then there exists a completely multiplicative function $\chi \colon \mathbb N \longrightarrow \{ \pm 1 \} $ with the same properties.
\end{lemma}

\begin{proof} 
 Arguing as in \cite[8.1]{S} we use the identity  
\begin{align*}
 \int_1^{\infty}q(x)x^{-s}dx -\sum_p \chi(p)p^{-s}\log p = \\ s\int_1^{\infty}\left( \int_1^xq(y)dy-\sum_{p\le x } \chi(p) \log p \right)x^{-s-1}dx . \end{align*} 
It suffices to show that there exists a $ \chi $ such that 
\begin{equation}\label{2140-e}
   r(x):= \int_1^x q(y)dy-\sum_{p\le x } \chi(p)\log p = O \bigl(x^{\frac{21}{40}} \log x \bigr).
\end{equation}
Let 
\begin{equation}\label{xj} x_0=2, x_{j+1}=x_j+x_j^{\frac{21}{40}}. \end{equation} 
It is enough then to establish \eqref{2140-e} at the sequence $ x= x_j$. Indeed, if it is satisfied for $ x = x_j $ then for $x \in [ x_j , x_{ j+1} ) $ we have
$$ r(x) = r(x_j) + \int_{x_j}^x q(y)dy-\sum_{x_j<p\le x } \chi(p)\log p .$$ The first term in rhs is $O(x^{ 21/40 } \log x )$ by assumption, \[ \int_{x_j}^xq(y)dy = O(x^{\frac{21}{40}}) \] by the boundedness of $ q$, and the rightmost term in the right hand side is trivially $ O(x^{21/40}\log x) $.
 
 It remains to choose $ \chi $ so that $ r(x_j ) = O\bigl(x_j^{21/40}\log x_j \bigr) $. In fact we are going to choose it so that $ r ( x_j ) = O ( \log x_j ) $. This is done by induction in $ j $. 
%
 
 Denote by $ c_1 \in [ -\pi , \pi ) $ the argument of the number
 \[ \rho_j := r( x_j ) + \int_{ x_j }^{ x_{j+1}} q ,\] and let $ k_1 \in \{ 0, \pm 1 \} $ be chosen so that $ | 2k_1\pi /3 - c_j | \le \pi/3 $. Let $ p_1 $ be the smallest prime in $ [ x_j , x_{ j+1 } ) $, and define $ \chi (p_1) = e^{ 2i k_1\pi/3  } $. Then we take the number $ \rho_j - \chi(p_1) \log p_1 $, define $ c_2 \in [ -\pi, \pi ) $ to be its argument, choose $ k_2 \in \{ 0, \pm 1 \} $ so that $ | 2k_2\pi /3 - c_2  | \le \pi/3 $, and let $ \chi ( p_2 ) = e^{ 2i k_2\pi/3  } $ with $ p_2 $ the smallest prime in $ [ x_j , x_{ j+1 } )$ larger than $ p_1 $, and so on through all primes in $ [x_j, x_{ j+1 } ) $. We have 
  \begin{align} r ( x_{j+1} ) = r( x_j ) + \int_{ x_j }^{ x_{j+1}} q - \sum_{ p \in [ x_j , x_{ j+1} ) } \chi ( p ) \log p  = \nonumber \\ \rho_j - e^{ 2i k_1\pi/3  } \log p_1 -  \sum_{ p \in ( p_1 , x_{ j+1} ) } \chi ( p ) \log p . \nonumber 
  \end{align} 
  Then by elementary trigonometry \begin{align} \left| \rho_j - e^{ 2i k_1\pi/3 } \log p_1 \right| =  \left| |\rho_j| - e^{ i(2 k_1\pi/3 - c_1)  } \log p_1 \right| \le \nonumber \\ \left\{ \begin{array}{rl} | \rho_j | -  (\log p_1 ) / 4 , &\log p_1 \le | \rho_j |/2 , \cr   3\log p_1, & \log p_1 > | \rho_j |/2. \end{array} \right. \label{step}\end{align} because $ 2 \cos ( 2k_1\pi /3 - c_j ) \ge 1  $, and $ | a - e^{ i\pi / 3 } b | \le a - b/4 $ for real $ a , b $ whenever $ 0 \le b \le a/2 $. Repeating this estimate with $ \rho_j^1 = \rho_j - e^{ 2i k_1\pi/3 } \log p_1 $ in the place of $ \rho_j $, we get 
  \[ \left| \rho_j^1 - e^{ 2i k_2\pi/3 } \log p_2 \right| \le \max\left\{ \left| \rho_j \right| - \frac{\log p_1 + \log p_2}4 , 3\log p_2 \right\} , \] and so on. We are now going to use the fact that the interval $[x_j, x_{j+1}) $ contains at least $ C_1 x_j^{21/40}/\log x_j $ primes (\cite{BH}, p. 562) for some constant $ C_1>0 $ independent of $ j $. It shows that eventually, after all primes on the interval $ [ x_j , x_{ j+1} ) $ are accounted for, the first of the two numbers over which the maximum is taken, $ \left| \rho_j \right| - (\log p_1 + \log p_2 + \dots )/4 $, is not greater than $ \left| \rho_j \right| - C_1 x_j^{ 21/40}/4 $. 
  The other number in the maximum is obviously estimated above by   $ 3 \log p_* \le 3 \log x_{ j+1 } $, $ p_* $ being the largest prime on $ [ x_j , x_{ j+1} ) $. Thus, 
  \[ | r ( x_{ j+1 } ) | \le \max \left\{ | \rho_j | -  C_1 x_j^{ 21/40 }/4  , 3\log x_{ j+1 }  \right\} \]
  On the other hand, 
  \[ |\rho_j | \le | r ( x_j ) | + o ( x_j^{ 21/40 } ) \]
 since $ q $ is assumed to vanish at infinity. Picking $ M $ large enough so that the $ o ( x^{ 21/40 } )$-term gets smaller than $  C_1 x_i^{ 21/40 } / 8 $ for $ i \le M $, we obtain that whenever $ j \ge M $ 
 \[ | r ( x _{ j+1 } ) | \le \max \left\{ | r ( x_j )| , 3\log x_{ j+1 } \right\} . \] 
 It follows that $ | r(x_{ j+1 } ) | = O (\log x_{ j+1 }) $, as required. For the implied constant one can take the maximum over $ | r ( x _i )|/ \log x_i $, $ i \le M $, and $ 3 $. This proves the unconditional part of the assertion.

Assuming the Riemann hypothesis we follow the same argument with $ x_{i+1}=x_i + 4x_i^{1/2} \log x_i $ instead of \eqref{xj}, and take into account that  the number of primes in the interval $ [ x , x + c \sqrt{x} \log x ] $ is estimated below by $ \sqrt x $ for all $ c > 3 $ \cite{Dudek}. This gives  \[  r ( x ) = O( \sqrt{x} \log^2 x ) \] which implies the required assertion in the conditional case.

The part of the assertion pertaining to the case of real $ q $ is proved similarly but easier -- in this case instead of \eqref{step} we are going to have simply 
\[\left| \rho_j - \chi ( p_1 ) \log p_1 \right| =  ||\rho_j | - \log p_1 | |  . \] 
 \end{proof}
 
 
Thus it remains to find an analytic function $g_1 $ in the halfplane 
$ \Re s > 1 $ of the form \[ g_1(s)=\int_1^{\infty}q(x)x^{-s}dx,\] with $ q $ vanishing as $ x\to +\infty $ which admits meromorphic extension to the halfplane $\Re s>21/40 $ with the prescribed poles and respective residues. 

\begin{lemma}
Let $g$ be an analytic function in the halfplane $ \Re z > 1 $ such that 
$\sup |z|^2 |g(z)| < \infty $. Then there exists a continuous function, $q$, $ q ( s ) = o ( 1 ) $, $ s\to +\infty $, such that \[ g(s)=\int_1^{\infty}q(s)x^{-s}dx, \; \Re s > 1 . \]
If additionally $ g ( s) $ is real for real $ s $, then $ q $ is real. 
\end{lemma}

\begin{proof}
Consider the function $h(t)=g(-it+1)$. The function $h$ belongs to the Hardy class $ H^2_+ $, hence the restriction $\left. h \right|_{\mathbb R } $ is the inverse Fourier transform of a certain function $ p \in L^2 ( \mathbb R ) $, vanishing on the negative real axis. Since $ \left.h \right|_{ \mathbb R } \in L^2 \cap L^1 $, the function $ p $ is the classical Fourier transform of $ h $, hence it is continuous and vanishes as $ x \to +\infty $
by the Riemann--Lebesgue lemma. The function $q(s):= p(\log s)$ then also vanishes as $ s \to +\infty $. Finally we have (recall that $s=-it+1$):  
\begin{align*}\int_1^{\infty}q(x)x^{-s} dx =\int_0^{\infty}q(e^y)e^{(1-s)y}dy =\int_0^{\infty}p(y)e^{ity}d y=h(t)=g(s) , \end{align*} as required.
Finally, if $ g (s) $ is real for real $ s $, then $ \overline {h ( t ) } = h ( -t ) $ and therefore $ p $ and $ q $ are real.
\end{proof}

\medskip
\bfseries 
\begin{center}
Mellin transforms with prescribed poles and residues in a strip
\end{center}
\mdseries

In this section we construct the function $ g $. Let $\alpha = 21/40 $ in the unconditional case, and $ \alpha = 1/2 $ if RH is satisfied. Assume first that the sets $ Z $ and $ P $ have no accumulation points at finite distance. In view of lemma 2, Theorem 2 for this case will be proven if we manage to find a function $g$ analytic in the halfplane $ \Re s > 1 $ and satisfying $\sup_{ \Re z > 1 } |g(z)|\, |z|^2< \infty $, which admits meromorphic extension to the halfplane $ \Re z > \alpha $ with the given poles and residues in the strip $\alpha < \Re z <1 $.

Notice first that, given a point $z_0$, $\alpha< \Re z_0< 1 $, and a number $ C > 0 $, one can choose $ n = n ( C , z_0 )$ large enough so that the function \[ g_{z_0}(z) =\frac{1}{(z-z_0)(z-z_0+1)^{2n}} \] 
has the following properties, 

(i) $|g_{z_0}(z)| <C$ for $ \Re z >1 $.

(ii) $ g_{ z_0 } $ is analytic in $\{ \Re z > \alpha \} $ except at $ z_0 $,  has a simple pole at $z_0$ with $ \operatorname{Res}_{ z_0 } g_{ z_0 } = 1 $. 

(iii) $|g_{z_0}(z)|<C$ for $|z-z_0|> 3 $, $ \Re z >\alpha$. 

\begin{lemma}
Let $ \Sigma $ be a subset of the strip $\alpha< \Re z < 1 $ having no accumulation points at finite distance, and $ m \colon \Sigma \longrightarrow \mathbb C \setminus \{ 0 \} $ be an arbitrary function. Then there exists a meromorphic function $ g $ in the halfplane $\Re z >\alpha$, whose set of poles coincides with $ \Sigma $, all poles are simple, $ \operatorname{Res}_z g =  m ( z ) $ for $ z \in \Sigma$, and \[\sup_{ \Re z > 1 } |g(z)| |z|^2 < \infty .\]
If additionally the set $ \Sigma $ is symmetric with respect to the real line, $ m $ is real valued, and $ m ( z ) = m ( \overline z ) $, then $ g(z) $ is real for real $ z $. 
\end{lemma}

Note that the function $ m $ needed for Theorem 2 is in fact integer valued. 

\begin{proof} Let $ G_1 $ be an arbitrary function analytic in the halfplane $\Re z>\alpha$, having no zeroes, real at real $ z $, and satisfying $ G_1 ( z ) = O ( |z|^{-2} ) $ as $ |z| \to \infty $. $G_1(z)=e^{-z}z^{-2}$ will do.

Fix an arbitrary enumeration of $ \Sigma $. Given a $p_i \in \Sigma $, define $ g_i $ to be the function $ g_{ p_i } $ satisfying properties (i)--(iii) with \[ C=\frac{|G_1(p_i)|}{ |m (p_i)| 2^{i+1}}. \] Let
\begin{equation}\label{g} g(z)= G_1(z) \sum_i m ( p_i) \frac{g_i(z)}{G_1(p_i)}. \end{equation} Let us first check that the series in the rhs converge absolutely at any point $ z \notin \Sigma $ in the halfplane $ \Re z > \alpha $, the convergence being uniform on compacts in $ \{ \Re z > \alpha \} \setminus \Sigma $, and thus the function $ g $ is meromorphic with simple poles at $ \Sigma $ and no other singularities.

Indeed let $z \notin \Sigma$, $ \Re z > \alpha $. Clearly $|z-p_i|\le 3 $ for at most finitely many $ i $'s. If $|z-p_k|> 3 $ for some $ k $, then \[ |g_k(z)| < \frac{|G_1(p_k)|}{|m ( p_k )| 2^{k+1}}, \] by (iii) from whence \[ \left| m ( p_k ) \frac{g_k(z)}{G_1(p_k)}\right|< 2^{-k-1} ,\] and the convergence is proven. The equality $ \operatorname{Res}_{ p_i } g = m ( p_i ) $ is obvious. It remains to notice that for $\Re z> 1 $ we have \[ \left| m ( p_i ) \frac{g_i(z)}{G_1(p_i)} \right|\le 2^{-i-1} ,\] hence $ | g ( z) | \le | G_1 ( z )| $, and thus $ \sup_{ \Re z \ge 1} |z|^2 |g(z)| < \infty $, as required.

The assertion about reality is immediate from \eqref{g}.
\end{proof}

Theorem 2 is thus proved in the partial case when the sets $ Z $ and $ P $ do not accumulate at finite distance. The general case is reduced to this one via dyadic decomposition of the strip in the same way as in \cite[Section 8.1]{S}.
%

\end{document}